\documentclass[10pt,a4paper,reqno]{amsart}
\usepackage{amsthm}
\usepackage{amsmath}
\usepackage{amssymb}
\usepackage{graphicx,color}

\usepackage[font=small]{caption}

\usepackage{multirow}
\usepackage[shortlabels]{enumitem}

\makeatletter
\theoremstyle{plain}
\newtheorem{thm}{Theorem}
  \theoremstyle{definition}
  
  \theoremstyle{remark}
  \newtheorem{rem}[thm]{Remark}
  \theoremstyle{plain}
  \newtheorem{prop}[thm]{Proposition}
  \theoremstyle{plain}
  \newtheorem{lem}[thm]{Lemma}
  \theoremstyle{plain}
  
 \theoremstyle{definition}
  
  \theoremstyle{remark}
  \newtheorem*{rem*}{Remark}

  \theoremstyle{definition}

\usepackage{amsfonts}
\usepackage{mathrsfs}

\newtheorem*{question*}{\it{QUESTION}}

\theoremstyle{plain}

\newcommand{\N}{\mathbb{N}}
\newcommand{\R}{{\mathbb{R}}}
\newcommand{\C}{{\mathbb{C}}}
\newcommand{\D}{{\mathbb{D}}}

\newcommand{\dd}{{\rm d}}
\newcommand{\ii}{{\rm i}}

\renewcommand{\Im}{\mathop\mathrm{Im}\nolimits}

\newcommand{\sn}{\mathop\mathrm{sn}\nolimits}
\newcommand{\cn}{\mathop\mathrm{cn}\nolimits}
\newcommand{\dn}{\mathop\mathrm{dn}\nolimits}

\newcommand{\SC}{\mathop\mathrm{sc}\nolimits}
\newcommand{\NC}{\mathop\mathrm{nc}\nolimits}
\newcommand{\DC}{\mathop\mathrm{dc}\nolimits}

\newcommand{\sd}{\mathop\mathrm{sd}\nolimits}
\newcommand{\zn}{\mathop\mathrm{zn}\nolimits}

\makeatother

\begin{document}

\graphicspath{{C:/Data/00Synchronized/Figures/2015/FS_Sn/}}

\title{On extremal properties of Jacobian elliptic functions with complex modulus}

\author{Petr Siegl}
\address[Petr Siegl]{
	Mathematisches Institut, 
	Universit\"{a}t Bern,
	Alpeneggstrasse 22,
	3012 Bern, Switzerland
	\& On leave from Nuclear Physics Institute ASCR, 25068 \v Re\v z, Czech Republic}
\email{petr.siegl@math.unibe.ch}

\author{Franti\v sek \v Stampach}
\address[Franti\v sek \v Stampach]{
	Mathematisches Institut, 
	Universit\"{a}t Bern,
	Alpeneggstrasse 22,
	3012 Bern, Switzerland
	\& Department of Mathematics, 
	Stockholm University,
	Kr\"{a}ftriket 5,
	SE - 106 91 Stockholm, Sweden
	}
\email{stampfra@fjfi.cvut.cz}



\date{December 16, 2015}

\begin{abstract}
A thorough analysis of values of the function $m\mapsto\sn(K(m)u\mid m)$ for complex parameter $m$ and $u\in (0,1)$ is given.
First, it is proved that the absolute value of this function never exceeds 1 if $m$ does not belong to the region in $\C$ determined by inequalities
$|z-1|<1$ and $|z|>1$. The global maximum of the function under investigation is shown to be always located in this
region. More precisely, it is proved that, if $u\leq1/2$, then the global maxim is located at $m=1$ with the value equal to $1$. While if $u>1/2$,
then the global maximum is located in the interval $(1,2)$ and its value exceeds $1$. In addition, more subtle extremal properties are studied numerically.
Finally, applications in a Laplace-type integral and spectral analysis of some complex Jacobi matrices are presented.
\end{abstract}

\maketitle


\section{Introduction and statement of the main result}

Jacobian elliptic functions $\sn(z\mid m)$, $\cn(z\mid m)$, $\dn(z\mid m)$ have been studied in depth, in particular, 
as functions of complex argument $z$; classical references are treatises \cite{akhiezer90, chandrasekharan85, lawden89, 
whittaker27} or handbooks \cite{abramowitz64, byrd71}.
However, the vast majority of investigations have been made when the parameter $m$ (or modulus $k$; see Section 
\ref{subsec:not} for details) is real and restricted to the interval $(0,1)$, which typically suffices in applications.
The articles of Walker \cite{walker_rslps03, walker_cmft09} and Schiefermayr \cite{schiefermayr_jmi12} belong to the 
few exceptions devoted directly to the study of Jacobian elliptic functions as functions of complex parameter $m$.

The aim of this paper is an investigation of properties of the function $m\mapsto\sn(K(m)u\mid m)$,
for $u\in\R$ and $m\in\C$, where $K(m)$ stands for the complete elliptic integral of the first kind.
Mainly, we focus on the localization of range of the function $m\mapsto|\sn(K(m)u\mid m)|$, and particularly, 
its extremal values.

The composition of the function $\sn$ with the complete elliptic integral $K$ is quite natural as one can observe, for example, from the relation with Jacobi's theta functions, see \eqref{eq:sn_eq_quotient_theta} below, or Eisenstein series \cite{dlmf22}; 
see also the discussion in \cite{carlson_siamjna83} and the applications in Section~\ref{sec:app}.

While $m\mapsto\sn(K(m)u\mid m)$ is an analytic function in the cut-plane $\C\setminus[1,\infty)$ with the
branch cut in $[1,\infty)$, its modulus is well defined, bounded and continuous function in the whole complex plane $\C$.
Our main result is the following theorem reflecting the non-trivial properties of $m \mapsto |\sn(K(m)u\mid m)|$. 
\begin{thm}\label{thm:summary}
	The following statements hold true.
	\begin{enumerate}[{\upshape i)}]
		\item For all $u\in(0,1)$ and $m\notin \{z \in \C \,:\, |z-1|<1 \wedge |z|>1 \}$, it holds 
		\begin{equation}\label{|sn|<1}
		|\sn(K(m)u\mid m)|<1.
		\end{equation}
		\item For $u\in(0,1/2]$ the function $m\mapsto|\sn(K(m)u\mid m)|$ has unique global maximum
		located at $m=1$ with the value equal to $1$, i.e.,
		\[
		|\sn(K(1)u\mid 1)|=1 \quad \mbox{ and } \quad |\sn(K(m)u\mid m)|<1 \; \mbox{ for all } m\neq1
		\]
		(where the value at $m=1$ is to be understood as the respective limit). 
		\item For $u\in(1/2,1)$, the function $m\mapsto|\sn(K(m)u\mid m)|$ has a global maximum
		located in the interval $(1,2)$ with the value exceeding $1$, i.e., 
		\[
		\max_{m\in\C}|\sn(K(m)u\mid m)|=|\sn(K(m^{*})u\mid m^{*})|>1 \; \mbox{ for some } m^{*}\in(1,2).
		\]
	\end{enumerate}
\end{thm}

Notice that there is no loss of generality in restricting $u$ to $(0,1)$, which is assumed throughout the whole paper, 
since the function $u\mapsto|\sn(K(m) u \mid m)|$ is $2$-periodic, $\sn(K(m) (2-u) \mid m)=\sn(K(m) u \mid m)$ and 
$\sn(0\mid m)=0$, $\sn(K(m)\mid m)=1$.
	
The proof of Theorem \ref{thm:summary}, consisting of several steps, is presented in Section \ref{sec:proof}. Besides the employment of many known 
properties (transformation and addition rules, derivatives w.r.t.~to both argument and parameter, asymptotic expansions 
as $m \to 0$ and $m \to 1$) of Jacobian elliptic functions, it heavily relies on the maximum modulus principle both for bounded and unbounded regions.

The situation for $m \in \{z \in \C \,:\, |z-1|<1 \wedge |z|>1 \}$ is more delicate and the region where \eqref{|sn|<1}
holds seems to have an interesting shape, see Figure \ref{fig:regions}. Moreover, the value of the global maximum of $|\sn(K(m)u\mid m)|$ on $(u,m) \in \R \times \C$ 
seems to exceed $1$ only slightly by approx.~$0.01$. For more details and other numerical results, see Section \ref{sec:num}.

Finally, in Section \ref{sec:app}, we explain an application of Theorem \ref{thm:summary} in the spectral analysis of certain family of semi-infinite complex Jacobi matrices
whose diagonal vanishes and off-diagonal is an unbounded periodically modulated sequence. In this case, the coupling constant presented in the modulated weight coincides with the modulus of a Jacobian elliptic function. The spectral problem for the corresponding family of orthogonal polynomials, with modulus being restricted to $(0,1)$, 
has been studied by Carlitz in \cite{carlitz_dmj60,carlitz_dmj61}. His motivation was, in turn, based on beautiful continued fraction formulas for Laplace transform of powers of Jacobian elliptic functions, which goes back to Stieltjes \cite{stieltjes93}. Theorem \ref{thm:summary} allows to treat the spectral analysis for general complex parameter; details are elaborated in a subsequent paper \cite{sieglstampach_inprep}.


\section{Preliminaries}

\subsection{Notations, auxiliary functions and analyticity properties}
\label{subsec:not}

As all the Jacobian elliptic functions depend on $k^{2}$ rather than the modulus $k$ itself, we primarily follow 
Milne-Thompson's notation using parameter $m$ instead of the modulus $k$, i.e. $m=k^{2}$ and $\sn(z\mid m)=\sn(z,k)$, 
etc; this is also the case for Abramowitz \& Stegun \cite[Chap.~16]{abramowitz64} as well as \cite{walker_rslps03, walker_cmft09}.
In the classical theory, the modulus $k$ is accompanied with the complementary modulus $k'$ which are related by 
equation $k^{2}+k'^{2}=1$;  in terms of parameter we use $m$ and $m_1:=1-m$.

We use the following notation for the upper and lower complex half-plane 
\[
\C_{\pm}:=\{z\in\C \mid \Im z \gtrless 0\},
\]
and the subsets of $\C$ related to the unit disks centered and $0$ and $1$ are denoted by
\begin{equation*}
\D:=\{z\in\C \: ,\: |z|<1\}, \quad \D_{1}:=\{z\in\C \: ,\: |z-1|<1\}, \quad  \overline{\D}:=\D\cup \partial \D, \quad \overline{\D}_1:=\D_1\cup \partial \D_1.
\end{equation*}
Further, throughout the whole paper, the principal part of the square root is always used.

Recall the complete elliptic integral of the first kind is defined by 
\[
  K(m)=\int_{0}^{1}\frac{\dd t}{\sqrt{(1-t^{2})(1-mt^{2})}}
\]
and it is analytic for $m\in\mathbb{C}\setminus[1,\infty)$. 
If not necessary, we will suppress the dependence on $m$ in the notation of~$K$.
$K(m)$ has a singularity in $m=1$ and a branch cut ain $(1,\infty)$. Limit values of $K(m)$, for $m>1$, exist if $m$ is approached from one of the half-planes
$\C_{+}$ or $\C_{-}$.
More precisely, one has the connection formula
\begin{equation}
 K\left(1/m\right)=m^{1/2}\left[K(m)\mp\ii K'(m)\right],
\label{eq:K_connect}
\end{equation}
where $K'(m)=K(m_1)$ and the upper or lower sign is taken according as $m\in\C_{+}$ or $m\in\C_{-}$, respectively. 
For references and an interesting discussion of this dichotomy in sign, see \cite{fettis_siamjma70}.

In the theory of elliptic functions, a conformal mapping known as the elliptic modular function 
$\lambda:\tau\mapsto m$ is investigated. For the inverse mapping, one has
\begin{equation}\label{tau.def}
\tau=\ii \frac{K'(m)}{K(m)}.
\end{equation}
The complementary variable to $\tau$ is denoted here by $\tau_{1}$
where $\tau_{1}=-1/\tau$. In addition, the nome $q$ is given by 
\begin{equation}\label{q.def}
q = e^{\ii\pi\tau} = e^{-\pi \frac{K'(m)}{K(m)}}.
\end{equation}
Properties of the function $\lambda$ are essential for understanding of 
dependencies and analyticity of mappings relating quantities $m\leftrightarrow\tau\leftrightarrow q$,
see \cite[Chap.~7, \S 8]{chandrasekharan85}.

Further recall that all Jacobian elliptic functions can be defined as a quotient of Jacobi's theta functions, 
see \cite[Chap.~2]{lawden89}. For the function $\sn$, one has
\begin{equation}
 \sn(K(m)u\mid m)=\frac{\theta_{3}(0,q)}{\theta_{2}(0,q)}\frac{\theta_{1}(u\pi/2,q)}{\theta_{4}(u\pi/2,q)}.
\label{eq:sn_eq_quotient_theta}
\end{equation}
All four theta functions are analytic 
functions in $m$ for $m\in\mathbb{C}\setminus[1,\infty)$, see \cite[Sec.~4]{walker_rslps03}.
Since all zeros of the theta functions are known explicitly, see \cite[Eq.~1.3.10]{lawden89}, it can be easily verified that theta functions in the denominator of the RHS in \eqref{eq:sn_eq_quotient_theta} 
never vanish for $u\in\R$ and $m\in\mathbb{C}\setminus[1,\infty)$. Consequently, for $u\in\mathbb{R}$, the function $m\mapsto\sn(K(m)u \mid m)$ is 
analytic in the cut-plane $\mathbb{C}\setminus[1,\infty)$.

Moreover, since $q(\overline{m})=\overline{q(m)}$ and $\theta_{l}(x,\overline{q})=\overline{\theta_{l}(x,q)}$, for $x\in\mathbb{R}$ and $l\in\{1,2,3,4\}$, it follows from \eqref{eq:sn_eq_quotient_theta} that, for $u\in\R$ and $m\in\mathbb{C}\setminus[1,\infty)$,
\begin{equation}
 \overline{\sn(K(m)u \mid m)}=\sn(K(\overline{m})u \mid \overline{m}).
\label{eq:sn_complex_conj} 
\end{equation}

\subsection{Asymptotic expansions for $m$ close to $0$ and $1$}

For later purposes, we first derive asymptotic expansions of $\sn(K(m)u \mid m)$, $\cn(K(m)u \mid m)$ and $\dn(K(m)u \mid m)$ for $m$  close to $0$ and $1$. 

Restricted to $m \in (0,1)$, Carlson and Todd proved in \cite{carlson_siamjna83} that, for $u\in(0,1)$, the function 
$m\mapsto\sn(K(m)u\mid m)$ is strictly increasing on $(0,1)$, cf.~\cite[Thm.~1]{carlson_siamjna83}, and found its  asymptotic behavior as $m\to0+$ and $m\to1-$, cf.~\cite[Sec.~4, Thm.~3]{carlson_siamjna83}. 

We require $m$ to be confined to the set $\C\setminus[1,\infty)$ only. While the proof for $m \to 0$ follows the usual strategy and straightforwardly reduces to known facts, the case when $m\to1$ is a bit more difficult and we present a method based of the Fourier series expansions of some Jacobian elliptic functions.

\begin{lem}\label{lem:asym}
For $u\in(0,1)$, we have the following asymptotic expansions
\begin{align}
\sn(K(m)u \mid m)&=\sin\left(\pi u/2\right)+O(m), \label{eq:sn_asympt_m_to_zero} \\
\cn(K(m)u \mid m)&=\cos\left(\pi u/2\right)+O(m), \label{eq:cn_asympt_m_to_zero} \\
\dn(K(m)u \mid m)&=1+O(m), \label{eq:dn_asympt_m_to_zero}
\end{align}
as $m\to0$, $m\in\C\setminus[1,\infty)$ 
and
\begin{align}
\sn(K(m)u\mid m)&=1-2^{1-4u} m_{1}^{u}+o\left(m_{1}^{u}\right), \label{eq:sn_asympt_m_to_1} 
\\
\cn(K(m)u\mid m)&=2^{1-2u}m_{1}^{u/2}+o\left(m_{1}^{u/2}\right), \label{eq:cn_asympt_m_to_1} 
\\
\dn(K(m)u \mid m)&=2^{1-2u}m_{1}^{u/2}+o\left(m_{1}^{u/2}\right), \label{eq:dn_asympt_m_to_1} 
\end{align}
as $m\to1$, $m\in\C\setminus[1,\infty)$. 
\end{lem}
\begin{proof}
The case $m\to0$: It follows that $q\to0$ and one can make use of the
relation between the three functions and Jacobi's theta functions \cite[Eqs.~2.1.1--3]{lawden89} together
with the well known expansions of the theta functions for $q\to0$, see \cite[Eqs.~2.1.14--17]{lawden89}.
This yields
\begin{align*}
\sn(K(m)u \mid m)&=\sin\left(\pi u/2\right)+O(q),\\
\cn(K(m)u \mid m)&=\cos\left(\pi u/2\right)+O(q),\\
\dn(K(m)u \mid m)&=1+O(q),
\end{align*}
for $m\to0$.  Of course, it is desirable to rewrite the RHSs in the last expansions in terms of the modulus $m$.
To do this, recall \cite[Eq.~2.1.12]{lawden89}
\[
m=16q\prod_{n=1}^{\infty}\left(\frac{1+q^{2n}}{1+q^{2n-1}}\right)^{\!8},
\]
from which one deduces
\begin{equation}
q=\frac{1}{16}m+\frac{1}{32}m^{2}+O\left(m^{3}\right), \quad m\to0,
\label{eq:q_asympt_m_to_zero}
\end{equation}
thus we arrive at the expansions \eqref{eq:sn_asympt_m_to_zero}--\eqref{eq:dn_asympt_m_to_zero} in the claim.

The case $m\to1$:
First, recall Fourier series expansions for functions $\SC$, $\NC$, and $\DC$, see \cite[Eqs.~16.23.7--9]{abramowitz64}, valid whenever $|\Im z/K|<2\Im\tau$, see \cite[Sec.~8.7]{lawden89},
\begin{align}
\SC(z \mid m)&=\frac{\pi}{2m_{1}^{1/2}K}\tan\left(\frac{\pi z}{2K}\right)
+\frac{2\pi}{m_{1}^{1/2}K}\sum_{n=1}^{\infty}\frac{(-1)^{n}q^{2n}}{1+q^{2n}}\sin\left(\frac{n\pi z}{K}\right)\!, \label{eq:sc_fourier}\\
\NC(z \mid m)&=\frac{\pi}{2m_{1}^{1/2}K}\sec\left(\frac{\pi z}{2K}\right)
-\frac{2\pi}{m_{1}^{1/2}K}\sum_{n=0}^{\infty}\frac{(-1)^{n}q^{2n+1}}{1+q^{2n+1}}\cos\left(\frac{(2n+1)\pi z}{2K}\right)\!, \label{eq:nc_fourier}\\
\DC(z \mid m)&=\frac{\pi}{2K}\sec\left(\frac{\pi z}{2K}\right)
+\frac{2\pi}{K}\sum_{n=0}^{\infty}\frac{(-1)^{n}q^{2n+1}}{1-q^{2n+1}}\cos\left(\frac{(2n+1)\pi z}{2K}\right)\!. \label{eq:dc_fourier}
\end{align}
Writing the trigonometric functions in \eqref{eq:sc_fourier}--\eqref{eq:dc_fourier} in terms of exponentials and using \eqref{q.def}, we arrive at formulas that are more convenient for our purposes:
\begin{align}
\SC(\tau K u \mid m)&=\frac{\ii\pi}{2m_{1}^{1/2}K}\frac{1-q^{u}}{1+q^{u}}
+\frac{\ii\pi}{m_{1}^{1/2}K}\sum_{n=1}^{\infty}(-1)^{n}\frac{q^{(2-u)n}}{1+q^{n}}\left(1-q^{2nu}\right)\!, \label{eq:sc_fourier_adj}\\ 
\NC(\tau K u \mid m)&=\frac{\pi}{m_{1}^{1/2}K}\frac{q^{u/2}}{1+q^{u}}
-\frac{\pi}{m_{1}^{1/2}K}\sum_{n=0}^{\infty}(-1)^{n}\frac{q^{(n+1/2)(2-u)}}{1+q^{2n+1}}
\left(1+q^{(2n+1)u}\right)\!, \label{eq:nc_fourier_adj}\\
\DC(\tau K u \mid m)&=\frac{\pi}{K}\frac{q^{u/2}}{1+q^{u}}
+\frac{\pi}{K}\sum_{n=0}^{\infty}(-1)^{n}\frac{q^{(n+1/2)(2-u)}}{1-q^{2n+1}}\left(1+q^{(2n+1)u}\right)\!, \label{eq:dc_fourier_adj}
\end{align}
where $u\in(-2,2)$. If $m\to0$, then $q\to0$, thus, from equalities \eqref{eq:sc_fourier_adj}--\eqref{eq:dc_fourier_adj},
we immediately deduce the asymptotic expansions formulas for $m\to0$ in terms of $q$ and subsequently, taking into account \eqref{eq:q_asympt_m_to_zero} 
and noticing that $K(m)\to\pi/2$ for $m\to0$, we obtain the asymptotic expansions
%
\begin{align}
\SC(\tau K u \mid m) &=\ii-\ii 2^{1-4u} m^{u}+o\left(m^{u}\right), \label{eq:sc_asympt_m_to_zero} \\
\NC(\tau K u\mid m)&=2^{1-2u}m^{u/2}+o\left(m^{u/2}\right), \label{eq:nc_asympt_m_to_zero} \\
\DC(\tau K u\mid m)&=2^{1-2u}m^{u/2}+o\left(m^{u/2}\right), \label{eq:dc_asympt_m_to_zero}
\end{align}
for $m\to0$ and $u\in(0,1)$.

Finally, by the Jacobi's imaginary transformation \cite[Eq.~16.20.1]{abramowitz64},
one gets
\[
\sn(K(m)u \mid m)=-\ii\SC\left(\ii K(m)u \mid m_1\right)=-\ii\SC\left(\tau_{1}K(m_1)u \mid m_1\right)
\]
where the last equality holds since for the complementary variable $\tau_1$ one has $\tau_{1}=-1/\tau=\ii K(m)/K(m_1)$. 
Note that if $m\to1$, then $m_{1}\to0$, thus to obtain \eqref{eq:sn_asympt_m_to_1}, it suffices to apply \eqref{eq:sc_asympt_m_to_zero}.
The remaining two expansions for $\cn(K(m)u \mid m)$ and $\dn(K(m)u \mid m)$ follow in an
analogous way; the corresponding Jacobi's imaginary transformation \cite[Eqs.~16.20.2--3]{abramowitz64} 
and expansions \eqref{eq:nc_asympt_m_to_zero}--\eqref{eq:dc_asympt_m_to_zero} are used.
\end{proof}
\section{Proof of the main result}
\label{sec:proof}

\subsection{Region $m \in \overline \D$}

We begin with analysis of $|\sn(K(m)u\mid m)|$ for $m$ in $\partial \D \setminus \{1 \}$ and subsequently extend 
the obtained inequality on $\overline \D \setminus \{1\}$.

\begin{lem}\label{lem:sn_mod_|m|=1}
For all $(u,m)\in(0,1)  \times \partial \D \setminus\{1\}$, we have
 \[
  |\sn(K(m)u\mid m)|<1.
 \]
\end{lem}

\begin{proof}
 Taking into account \eqref{eq:sn_complex_conj}, it is sufficient to verify the statement for $m=e^{4\ii\theta}$ with $\theta\in(0,\pi/4]$.
  First, we apply the ascending Landen transformation \cite[Eq.~16.14.2]{abramowitz64} on $\sn(K(m)u \mid m)$ and receive the equality
%
\[
  \sn\left(K\left(e^{4\ii\theta}\right)u \mid e^{4\ii\theta}\right)=\frac{e^{-\ii\theta}}{\cos\theta}
  \frac{\sn\left(K\left(e^{4\ii\theta}\right)v\mid \cos^{-2}\theta\right)\cn\left(K\left(e^{4\ii\theta}\right)v\mid \cos^{-2}\theta\right)}
  {\dn\left(K\left(e^{4\ii\theta}\right)v\mid \cos^{-2}\theta\right)},
 \]
where $v=ue^{\ii\theta}\cos\theta$. Next, by the double argument formula \cite[Eq.~16.18.5]{abramowitz64}, we arrive at
 \[
  \sn^{2}\left(K\left(e^{4\ii\theta}\right)u \mid e^{4\ii\theta}\right)=e^{-2\ii\theta}
  \frac{1-\dn\left(2K\left(e^{4\ii\theta}\right)v\mid \cos^{-2}\theta\right)}{1+\dn\left(2K\left(e^{4\ii\theta}\right)v\mid \cos^{-2}\theta\right)}
 \]
and by the Jacobi's real transformation \cite[Eq.~16.11.4]{abramowitz64}, we get
  \[
  \sn^{2}\left(K\left(e^{4\ii\theta}\right)u \mid e^{4\ii\theta}\right)=e^{-2\ii\theta}
  \frac{1-\cn\left(2K\left(e^{4\ii\theta}\right)e^{\ii\theta}u \mid \cos^{2}\theta\right)}
  {1+\cn\left(2K\left(e^{4\ii\theta}\right)e^{\ii\theta}u \mid \cos^{2}\theta\right)}.
 \]
 Finally, the application of formula 
	\begin{equation}
	K\left(e^{4\ii\theta}\right)=\frac{1}{2}e^{-\ii\theta}
	\left[K\left(\cos^{2}\theta\right)+\ii K\left(\sin^{2}\theta\right)\right], \; \mbox{ for } \theta\in(0,\pi/4],
	\label{eq:K_upper_semi_circle}
	\end{equation}
see \cite{fettis_siamjma70} and references therein, yields
 \begin{equation}
    \sn^{2}\left(K\left(e^{4\ii\theta}\right)u \mid e^{4\ii\theta}\right)=e^{-2\ii\theta}
  \frac{1-\cn\left(\left[K\left(\cos^{2}\theta\right)+\ii K\left(\sin^{2}\theta\right)\right]u \mid \cos^{2}\theta\right)}
  {1+\cn\left(\left[K\left(\cos^{2}\theta\right)+\ii K\left(\sin^{2}\theta\right)\right]u \mid \cos^{2}\theta\right)}
 \label{eq:sn_sq_unit_circ}
 \end{equation}
 where $\theta\in(0,\pi/4]$.
 By \cite[Eq.~16.21.3]{abramowitz64}, we have
 \[
  \cn\left(\left[K\left(\cos^{2}\theta\right)+\ii K\left(\sin^{2}\theta\right)\right]u\mid \cos^{2}\theta\right)=
  \frac{c\cdot c_{1}-\ii s\cdot s_{1}\cdot d\cdot d_{1}}{c_{1}^{2}+\cos^{2}(\theta)\,s^{2}\cdot s_{1}^{2}}
 \]
 where the abbreviations $s=\sn\left(K\left(\cos^{2}\theta\right)u \mid \cos^{2}\theta\right)$, 
 $s_{1}=\sn\left(K\left(\sin^{2}\theta\right)u \mid \sin^{2}\theta\right)$, etc., are used.
 With the aid of the last formula and identity \eqref{eq:sn_sq_unit_circ} one deduces that
 \[
 \left|\sn\left(K\left(e^{4\ii\theta}\right)u \mid e^{4\ii\theta}\right)\right|<1
 \]
 if and only if $\left(c_{1}^{2}+\cos^{2}(\theta)\,s^{2}\cdot s_{1}^{2}\right)c\cdot c_{1}>0$. However, the latter 
 inequality is clearly true since all the functions $s$, $s_1$, $c$, $c_1$ are positive for $u\in(0,1)$.
\end{proof}

\begin{prop}\label{prop:sn_mod_strict_ineq}
 For all $(u,m)\in (0,1) \times \overline{\D}\setminus\{1\}$, we have
 \begin{equation}
 |\sn(K(m)u\mid m)|<1.
 \label{eq:sn_mod_strict_ineq}
 \end{equation}
\end{prop}

\begin{proof}
 By Lemma~\ref{lem:sn_mod_|m|=1}, the statement holds for $m\in \partial \D \setminus\{1\}$, hence it suffices  to verify the inequality for $m\in \D$.  Note that by \eqref{eq:sn_asympt_m_to_1},
 \[
 \lim_{\substack{m\to1 \\ m\in\overline{\D}\setminus\{1\}}}\sn(K(m)u\mid m)=1.
 \]
 Thus, the function $m\mapsto\sn(K(m)u\mid m)$ extends continuously to $\overline{\D}$.
 
For $m\in \D$, we denote $f(m):=\sn(K(m)u\mid m)$. 
 Recall that the M\"{o}bius transform
 \[
  \varphi_{m}(z)=\frac{m-z}{1-\overline{m}z}
 \]
 is a conformal self-map of the unit disk $\D$. Moreover, $\varphi_{m}$ maps the boundary $\partial \D$ into itself. Hence
 $g:=f\circ\varphi_{m}$ is analytic on $\D$, continuous on the closure $\overline{\D}$, and $g(0)=f(m)$.
By the Cauchy integral formula, one has
 \[
  |f(m)|=|g(0)|\leq\frac{1}{2\pi}\int_{0}^{2\pi}\left|g\left(re^{\ii\theta}\right)\right|\dd\theta,
 \]
 for any $r\in(0,1)$. By the Lebesgue's dominated convergence theorem, one can send $r\to1-$ in the above integral
and thereby obtain the inequality
 \[
  |f(m)|\leq\frac{1}{2\pi}\int_{0}^{2\pi}\left|g\left(e^{\ii\theta}\right)\right|\dd\theta.
 \]
The statement follows, since by Lemma \ref{lem:sn_mod_|m|=1}, $\left|g\left(e^{\ii\theta}\right)\right|<1$ for almost all $\theta\in(0,2\pi)$
 and hence
 \[
  \int_{0}^{2\pi}\left|g\left(e^{\ii\theta}\right)\right|\dd\theta<2\pi.
  \qedhere
 \]
\end{proof}

\subsection{Region $m \in\partial\D_{1}$}

Here we show the values of $|\sn(K(m)u\mid m)|$ remain below $1$ also for $m$ on the boundary of $\D_{1}$. 
For this purpose, we need the following auxiliary inequality.

\begin{lem}\label{lem:aux_lem}
For all $(x,m) \in (0,2K) \times [0,1/2]$, it holds
 \begin{equation}
  4\dn^{2}(x\mid m)\left(1+\cn(x\mid m)\right)-\sn^{2}(x\mid m)\left(1-\cn(x\mid m)\right)>0.
 \label{eq:aux_ineq_in_lem}
 \end{equation}
\end{lem}

\begin{proof}
 Denote the function on the LHS of inequality \eqref{eq:aux_ineq_in_lem} by $f(x)$. We show that $f$ is strictly
 decreasing on $(0,2K)$ and since $f(2K)=0$, the statement will be proved.
 
 We abbreviate $s:=\sn(x\mid m)$, etc. By using the formulas for the derivatives of Jacobian elliptic functions, see \cite[Eqs.~16.16.1--3]{abramowitz64}, and
 elementary identities $d^{2}=1-m+mc^{2}$, $s^{2}=1-c^{2}$, one computes $f'(x)=-s\cdot d \ g(x)$ where 
\[
 g(x)=2(4m+1)c+2(4m-1)c^{2}+4d^{2}+s^{2}=5-4m+2(4m+1)c+3(4m-1)c^{2}.
\]
Since $s\cdot d>0$ on $(0,2K)$, it suffices to show that $g(x)>0$ for $x\in(0,2K)$.
We distinguish to cases when $x\in(0,K]$ and $x\in(K,2K)$.

i) Let $x\in(0,K]$. Since
\[
 2(4m+1)c\geq 2c \quad \mbox{ and } \quad 3(4m-1)c^{2}\geq -3c^{2}\geq -3c,
\]
one gets $g(x)\geq 4(1-m)+1-c>0$ for all $m\in[0,1)$.

ii) Let $x\in(K,2K)$. We show that $g$ is strictly decreasing on $(K,2K)$ which together with identity $g(2K)=0$ implies positivity of $g$. Differentiating $g$ yields
\[
 g'(x)=-2s\cdot d \left[4m+1+3(4m-1)c\right].
\]
Since $c>-1$, the expression in the square brackets in the last formula can be estimated as
\[
 4m+1+3(4m-1)c>4(1-2m),
\]
hence $g'(x)<0$ for all $x\in(K,2K)$ and $m\in[0,1/2]$.
\end{proof}

\begin{lem}\label{lem:sn_mod_strict_ineq_D1}
 For all $(u,m) \in (0,1) \times \partial \D_1$, we have
 \begin{equation}
  |\sn(K(m)u\mid m)|<1.
 \end{equation}
\end{lem}

\begin{proof}
 Taking into account relation \eqref{eq:sn_complex_conj}, it suffices 
 to prove the statement for $m=1-e^{4\ii\theta}$ for $\theta\in[-\pi/4,0]$.
  Applying the descending Landen transformation \cite[Eq.~16.12.2]{abramowitz64} on $\sn\left(K\left(1-e^{4\ii\theta}\right)u \mid 1-e^{4\ii\theta}\right)$
we obtain
 \[
  \sn\left(K\left(1-e^{4\ii\theta}\right)u \mid 1-e^{4\ii\theta}\right)=\frac{e^{-\ii\theta}}{\cos\theta}
  \frac{\sn\left(K'\left(e^{4\ii\theta}\right)v\mid -\tan^{2}\theta\right)}
  {1-\ii \tan(\theta) \sn^{2}\left(K'\left(e^{4\ii\theta}\right)v\mid -\tan^{2}\theta\right)},
 \]
where $v=u e^{\ii \theta} \cos \theta$.
Next, by applying transform \cite[Eq.~16.10.2]{abramowitz64} together with equation 
 \begin{equation}
K'\left(e^{4\ii\theta}\right)=e^{-\ii\theta}K\left(\sin^{2} \theta\right) \; \mbox{ for } \theta\in[0,\pi/4),
 \label{eq:K'_unit_circle}
 \end{equation}
 see \cite{fettis_siamjma70} and references therein,
one gets
 \[
  \sn\left(K\left(1-e^{4\ii\theta}\right)u \mid 1-e^{4\ii\theta}\right)=e^{-\ii\theta}
  \frac{\sd\left(K\left(\sin^{2}\theta\right)u\mid \sin^{2}\theta\right)}
  {1-\ii\sin\theta\cos\theta\sd^{2}\left(K\left(\sin^{2}\theta\right)u\mid \sin^{2}\theta\right)}
 \]
 which holds true for $\theta\in[-\pi/4,0]$.
 
 Further, by using identities for double arguments \cite[Sec.~16.18]{abramowitz64}, one verifies that
 \begin{eqnarray*}
  \left|\sn\left(K\left(1-e^{4\ii\theta}\right)u \mid 1-e^{4\ii\theta}\right)\right|^{4}
  =\frac{\sn^{2}\left(2K\left(\sigma\right)u\mid \sigma\right)}{4\dn^{2}\left(2K\left(\sigma\right)u\mid \sigma\right)}
  \frac{1-\cn\left(2K\left(\sigma\right)u\mid \sigma\right)}{1+\cn\left(2K\left(\sigma\right)u\mid \sigma\right)}
 \end{eqnarray*}
 where $\sigma:=\sin^{2}\theta$. Consequently, 
 \[
 \left|\sn\left(K\left(1-e^{4\ii\theta}\right)u \mid 1-e^{4\ii\theta}\right)\right|<1
 \]
 if and only if
 \begin{equation*}
  4d^{2}(1+c)-s^{2}(1-c)>0
 \end{equation*}
  where the abbreviations $s=\sn\left(2K\left(\sigma\right)u \mid \sigma\right)$, etc., are used.
  This is, however, true for all $u\in(0,1)$ by Lemma \ref{lem:aux_lem}, since $\sigma\in[0,1/2]$ for  $\theta\in[-\pi/4,0]$.  
\end{proof}

\subsection{Region $m \notin \overline{\D}$ }

We take a closer look to the values of $\sn(K(m)u \mid m)$ for $|m|>1$.
If in addition $m\in\C_{+}$, then, by \eqref{eq:K_connect}, one has
\begin{equation}
 K(m)=\mu^{1/2}\left[K\left(\mu\right)+\ii K'\left(\mu\right)\right]\!
\end{equation}
where $\mu=m^{-1}$. The last formula and the Jacobi's real transformation \cite[Eq.~16.11.2]{abramowitz64} yield
\[
 \sn(K(m)u \mid m)=\mu^{1/2}\sn\left(\left[K(\mu) + \ii K'(\mu)\right]u \mid \mu\right).
\]
Furthermore, by the addition formula \cite[Eq.~16.17.1]{abramowitz64} and Jacobi's
imaginary transformation \cite[Sec.~16.20]{abramowitz64}, one arrives at the expression
\begin{equation}
 \sn(K(m)u \mid m)=\mu^{1/2}\,\frac{s\cdot d_{1}+\ii c \cdot d \cdot s_1 \cdot c_1}{1-d^{2}\cdot s_{1}^{2}}
\label{eq:sn_outside_unit_circle}
\end{equation}
where $s=\sn(K(\mu)u\mid \mu)$, $s_{1}=\sn(K(\mu_{1})u\mid \mu_{1})$, etc., and $\mu_{1}=1-\mu$.
Identity \eqref{eq:sn_outside_unit_circle} holds for $m\in\C_{+}$. If $m\in\C_{-}$, one has to change the sign at $\ii$ in the nominator in the RHS of \eqref{eq:sn_outside_unit_circle}.

For $m<1$, the function $\sn(K(m)u \mid m)$ is real-valued. The interval $(1,\infty)$ is a branch cut of $\sn(K(m)u \mid m)$ and the limiting values for $m>1$ are
\[
\sn(K(m)u \mid m)=\mu^{1/2}\,\frac{s\cdot d_{1}\pm\ii c \cdot d \cdot s_1 \cdot c_1}{1-d^{2}\cdot s_{1}^{2}}
\]
with the same abbreviation as above. The upper sign applies, if $m$ is approached from the upper-plane $\C_{+}$,
while the lower sign applies, if $m$ is approached from the lower-plane $\C_{-}$. Note, however, that the absolute 
value of $\sn(K(m)u \mid m)$ extends continuously on $(1,m)$ and one has
\begin{equation}
|\sn(K(m)u \mid m)|^{2}=\mu\frac{s^{2}\cdot d_{1}^{2}+ c^{2} \cdot d^{2} \cdot s_{1}^{2} \cdot c_{1}^{2}}
{\left(1-d^{2}\cdot s_{1}^{2}\right)^{2}} \; \mbox{ for } m>1.
\label{eq:sn_m>1}
\end{equation}
All in all, one concludes the function $m\mapsto|\sn(K(m)u \mid m)|$ extends continuously on the whole plane $\C$.
In addition to the continuity of $m\mapsto|\sn(K(m)u \mid m)|$, we will show that this function is also bounded.

\begin{prop}\label{prop:sn_lim_infty}
 For $u\in(0,1)$, it holds
  \[
  \lim_{\substack{m\to\infty \\ m\in\mathbb{C}\setminus[1,\infty)}}\sn(K(m)u\mid m)=0.
  \]
 In particular, function $m\mapsto\sn(K(m)u\mid m)$ is bounded on $\C\setminus[1,\infty)$.
\end{prop}

\begin{proof}
 We will prove that $\sn(K(m)u\mid m)\to0$ for $m\to\infty$ and $\Im m>0$; by the reflection relation \eqref{eq:sn_complex_conj}, one obtains the same result for $m\to\infty$ and $\Im m<0$. Then the statement follows by the analyticity of $m\mapsto\sn(K(m)u\mid m)$
 on $\C\setminus[1,\infty)$.

 Let $\Im m>0$. If $m\to\infty$, then $\mu=m^{-1}\to0$ and $\mu_{1}\to1$. By using formula \eqref{eq:sn_outside_unit_circle}
 and applying the asymptotic expansions \eqref{eq:sn_asympt_m_to_zero}--\eqref{eq:dn_asympt_m_to_zero}
 together with \eqref{eq:sn_asympt_m_to_1}--\eqref{eq:dn_asympt_m_to_1}, one obtains
 \[
  \sn(K(m)u\mid m)=\ii e^{-\ii\pi u/2}\,2^{2u-1}\,m^{(u-1)/2}\left(1+o(1)\right), \; \mbox{ for } m\to\infty.
 \]
 By taking into account the assumption $u\in(0,1)$, one verifies the statement.
\end{proof}

Finally, we investigate the region $m \in (1,\infty)$. To do this, an auxiliary result is needed.

\begin{lem}\label{lem:ineq}
 For any $u\in(0,1)$, the function
 \[
  \varphi(u,\mu):=\DC^{2}(K(\mu_1)u\mid\mu_1)-\DC^{2}(K(\mu)u\mid\mu)
 \]
 is strictly increasing in $\mu$ on $(0,1)$.
 In addition, if $u\in(0,1/2]$, then $\varphi(u,\mu)<1$ for all $\mu\in(0,1)$,
 while if $u\in(1/2,1)$, then there exists a unique $\tilde{\mu}\in(1/2,1)$ such that $\varphi(u,\tilde{\mu})=1$.
\end{lem}

\begin{proof}
We used the abbreviated notations  $s=\sn(K(\mu)u\mid \mu)$, $s_1=\sn(K(\mu_1)u\mid \mu_1)$, etc., and also $z=\zn\left(K(\mu)u\mid \mu\right)$ which is the Jacobi's zeta function, see \cite[Sec.~3.6]{lawden89}.

By using formulas \cite[Eqs.~710.00 and 710.60]{byrd71}, one obtains
\begin{equation}
\frac{\rm d}{\rm d \mu}\DC(K(\mu)u\mid\mu)=-\frac{s \cdot z}{2\mu c^{2}}.
\label{eq:dc_der_modul}
\end{equation}
Then \eqref{eq:dc_der_modul} leads to
\[
\frac{\partial}{\partial\mu}\varphi(u,\mu)=\frac{d_{1} \cdot s_{1} \cdot z_{1}}{\mu_{1}c_{1}^{3}}+\frac{d \cdot s \cdot z}{\mu c^{3}}>0
\]
for all $\mu\in(0,1)$ and $u\in(0,1)$ since the Jacobi's zeta function $\zn(u\mid \mu)$ has positive values
for $u\in(0,K)$, see \cite{carlson_siamjna83}. Consequently, the function $\varphi(u,\cdot)$ is strictly increasing on $(0,1)$ for any $u\in(0,1)$.

The asymptotic formulas \eqref{eq:sn_asympt_m_to_zero}--\eqref{eq:dn_asympt_m_to_1} yield
\[
\lim_{\mu\to1-}\varphi(u,\mu)=\tan^{2}\left(\pi u/2\right).
\]
Thus, for $u\in(0,1/2]$, $\varphi(u,\mu)<1$ for all $\mu\in(0,1)$. While, for $u\in(1/2,1)$, there exists a unique solution $\tilde{\mu}\in(0,1)$ of the equation $\varphi(u,\mu)=1$.
Since clearly $\varphi(u,1/2)=0$, $\varphi(u,\mu)\leq0$ for $\mu\leq1/2$ and hence $\tilde{\mu}>1/2$.
\end{proof}

\begin{prop}\label{prop:sn_mod_m>1}
The following statements hold true.
\begin{enumerate}[{\upshape i)}]
	\item For all $(u,m) \in(0,1/2] \times (1,\infty) \cup (0,1) \times [2,\infty)$, we have  
	$$|\sn(K(m)u \mid m)|<1.$$
	\item Let $u\in(1/2,1)$ and $\tilde{m}=\tilde{m}(u)\in(1,2)$ be the unique solution of the equation 
	\[
	\DC^{2}\left(K\left(1-m^{-1} \right) u\mid 1-m^{-1}\right)-\DC^{2}\left(K\left(m^{-1}\right) u\mid m^{-1}\right)=1.
	\]
	Then, for all $(u,m) \in (1/2,1) \times (1,\tilde{m})$, we have
	\[
	|\sn(K(m)u \mid m)|>1,
	\]
	while, for all $(u,m) \in (1/2,1) \times (\tilde{m},2)$, we have
	\[
	|\sn(K(m)u \mid m)|<1.
	\]

\end{enumerate}
\end{prop}

\begin{proof}
We use the abbreviated notations $s=\sn(K(\mu)u\mid \mu)$, $s_1=\sn(K(\mu_1)u\mid \mu_1)$, etc., as in the proof of Lemma \ref{lem:ineq}. Recall also that we denoted $\mu=m^{-1}$.

First, we rewrite identity \eqref{eq:sn_m>1}.
By writing $\mu s^{2}=1-d^{2}$, $d_{1}^{2}=1-\mu_{1}s_{1}^{2}$, $\mu c^{2}=d^{2}-\mu_{1}$, and finally 
$c_{1}^{2}=1-s_{1}^{2}$, one verifies that
\[
\mu\left(s^{2}d_{1}^{2}+ c^{2}d^{2}s_{1}^{2}c_{1}^{2}\right)=(1-d^{2} \cdot s_{1}^{2})\left(1-d^{2} \cdot s_{1}^{2}-d^{2}(1-s_{1}^{2})+(d^{2}-\mu_{1})s_{1}^{2}\right).
\]
Consequently, from identity \eqref{eq:sn_m>1}, one has
\[
  |\sn(K(m)u \mid m)|^{2}=1-\frac{d^{2}(1-s_{1}^{2})-(d^{2}-\mu_{1})s_{1}^{2}}{1-d^{2} \cdot s_{1}^{2}}
  =1-\frac{d^{2} \cdot c_{1}^{2}-\mu c^{2} \cdot s_{1}^{2}}{1-d^{2} \cdot s_{1}^{2}}.
\]
Next, by using \cite[Eq.~16.~9.~3]{abramowitz64}, the last identity can be further rewritten into the form
\begin{equation}
 |\sn(K(m)u \mid m)|^{2}=1-\frac{c^{2}\cdot c_{1}^{2}}{1-d^{2} \cdot s_{1}^{2}}\left(1-\varphi(u,\mu)\right)
\label{eq:sn_abs_simpl}
\end{equation}
where the notation of Lemma \ref{lem:ineq} is used.
Noticing that the factor
\[
 \frac{c^{2}\cdot c_{1}^{2}}{1-d^{2} \cdot s_{1}^{2}}>0
\]
for all $(m,\mu) \in (0,1)^2$, it follows from the formula \eqref{eq:sn_abs_simpl} that the statement can be obtained by inspection of the values
of $\varphi(u,\mu)$. This, however, is treated in Lemma \ref{lem:ineq}. 
Indeed, by Lemma \ref{lem:ineq}, for all $(u,\mu) \in (0,1/2] \times (0,1) \cup (0,1) \times (0,1/2]$, 
it holds that $\varphi(u,\mu)<1$, so \eqref{eq:sn_abs_simpl} then implies the validity of claim (i). 

On the other hand, if $u\in(1/2,1)$, according to Lemma \ref{lem:ineq}, there exists a unique $\tilde{\mu}\in(1/2,1)$ such that $\varphi(u,\mu)=1$. 
Hence by the monotonicity of $\varphi(u,\cdot)$, $\varphi(u,\mu)<1$ for $\mu<\tilde{\mu}$ and $\varphi(u,\mu)>1$ for $\mu>\tilde{\mu}$. This yields claim (ii) with $\tilde{m}=\tilde{\mu}^{-1}$.
\end{proof}

\begin{rem}
By applying \cite[Eq.~16.~9.~3]{abramowitz64}, identity \eqref{eq:sn_abs_simpl} can be further simplified to the form
 \[
  |\sn(K(m)u \mid m)|^{2}=1+\frac{\mu\SC^{2}(K(\mu_1) u \mid\mu_1)-\DC^{2}(K(\mu) u\mid\mu)}{1+\SC^{2}(K(\mu)u\mid\mu)\DC^{2}(K(\mu_1)u\mid\mu_1)}.
 \]
 This formula could be useful for more subtle investigation of properties of $|\sn(K(m)u \mid m)|$ for $m>1$. First of all, it follows from 
 claim (ii) of Proposition \ref{prop:sn_mod_m>1} that the function $m \mapsto |\sn(K(m)u \mid m)|$, with $u\in(1/2,1)$, has a maximum located in 
 $(1,\tilde{m})$. However, it is not clear whether this maximum is unique, although there is a strong numerical evidence for this statement.
\end{rem}

Now everything is prepared to prove the main theorem.

\subsection{Proof of Theorem \ref{thm:summary}}

Proof of assertion (i): Let $u\in(0,1)$. We know already that 
\begin{equation}\label{ineq:sn.1}
   |\sn(K(m)u\mid m)|<1,
\end{equation}
for all $ m\in\overline{\D}\setminus\{1\} \cup \partial\left(\D\cup \D_{1}\right)\cup(2,\infty)$, see Proposition \ref{prop:sn_mod_strict_ineq}, Lemma \ref{lem:sn_mod_strict_ineq_D1} and Proposition \ref{prop:sn_mod_m>1}, part~(ii).

Since $\partial \left(\D\cup \D_{1}\right)$ is a compact set
  and taking into account that Proposition \ref{prop:sn_lim_infty} together with the continuity of 
  $m\mapsto|\sn(K(m)u\mid m)|$ implies $|\sn(K(m)u\mid m)|\to0$ as $m\to\infty$, there exists a constant 
  $c<1$ such that, for all $m\in\partial\left(\D\cup \D_{1}\right)\cup(2,\infty)$,
  \begin{equation}\label{ineq:sn.c}
  |\sn(K(m)u\mid m)|\leq c.
  \end{equation}
  Finally, since $m\mapsto\sn(K(m)u\mid m)$ is analytic in the region $\C\setminus\left(\overline{\D\cup \D_{1}}\cup(2,\infty)\right)$,
  continuous to the boundary and bounded in this region, the maximum modules principle for unbounded 
  regions, see, for example, \cite[Thm.~15.1]{bak10}, implies that \eqref{ineq:sn.c} holds for all $m \in \C\setminus\left(\D\cup \D_{1}\right)$.
 
Proof of assertion (ii): Let $u\in(0,1/2]$. By \eqref{eq:sn_asympt_m_to_1}, $|\sn(K(1)u\mid 1)|=1$.
Further, according to Proposition \ref{prop:sn_mod_m>1}, part~(i), and 
Lemmas \ref{lem:sn_mod_|m|=1} and \ref{lem:sn_mod_strict_ineq_D1}, we know that \eqref{ineq:sn.1} holds
for all $m\in\partial\left((\D_{1}\setminus \D)\cap\C_{+}\right) \setminus\{1\}$.
Since $m \mapsto \sn(K(m)u\mid m)$ is analytic in $(\D_{1}\setminus \D)\cap\C_{+}$ and continuous to the boundary, the maximum modulus principle implies that, for all $m\in\overline{(\D_{1}\setminus \D)\cap\C_{+}}$,
\[
  |\sn(K(m)u\mid m)| \leq 1.
\]
However, the last inequality is strict, which can be verified by a standard procedure based on the averaging property of analytic functions, i.e., a function $f$ analytic in a region $G\subset\C$ satisfies
\[
 f(z_{0})=\frac{1}{2\pi}\int_{0}^{2\pi}f\left(z_{0}+re^{\ii\theta}\right)\dd\theta
\]
for any $z_{0}$ and $r>0$ such that $\{z\in\C\mid |z-z_{0}|\leq r\}\subset G$. 
The rest now follows from \eqref{eq:sn_complex_conj} and the already proved assertion (i).

Proof of assertion (iii): Let $u\in(1/2,1)$. By Proposition \ref{prop:sn_mod_m>1} part~(iii), we know
 that there is a global maximum of the function $m\mapsto|\sn(K(m)u\mid m)|$ restricted to interval $[1,2]$ located at some
 point $m^{*}\in(1,2)$ and $|\sn(K(m^{*})u\mid m^{*})|>1$. Since for all $m\in\left(\partial(\D_{1}\setminus \D)\right)
 \cap\C_{+}$ we already known that $|\sn(K(m)u\mid m)|<1$. Thus, by the maximum modulus principle, $m^{*}$ is the global
 maximum of $m\mapsto|\sn(K(m)u\mid m)|$ restricted to $\overline{(\D_{1}\setminus \D)\cap\C_{+}}$. Now it suffices to take into
 account \eqref{eq:sn_complex_conj} and the already proved assertion (i).
\qed

\section{Numerical computations and concluding remarks}
\label{sec:num}

Theorem \ref{thm:summary} shows that the extremal properties of $m \mapsto |\sn(K(m)u\mid m)|$ are interesting particularly in the region $\D_{1}\setminus \D$ and the interval $[1,2]$; recall we assume $u\in(0,1)$. 
To simplify notations, we denote
\begin{equation}
\sigma(u,m):=|\sn(K(m)u\mid m)|.
\end{equation}
For different values of $u \in (0,1)$, the region in $\C$ where $\sigma(u,\cdot) \geq 1$ is plotted in Figure \ref{fig:regions}. As numerics suggests, there is exactly one maximum of $\sigma(u,\cdot)$ at $m^*(u)\in[1,2]$ 
and its position is indicated in Figure \ref{fig:regions} as well. 
\begin{figure}[htb!]
	\includegraphics[width= 0.31 \textwidth]{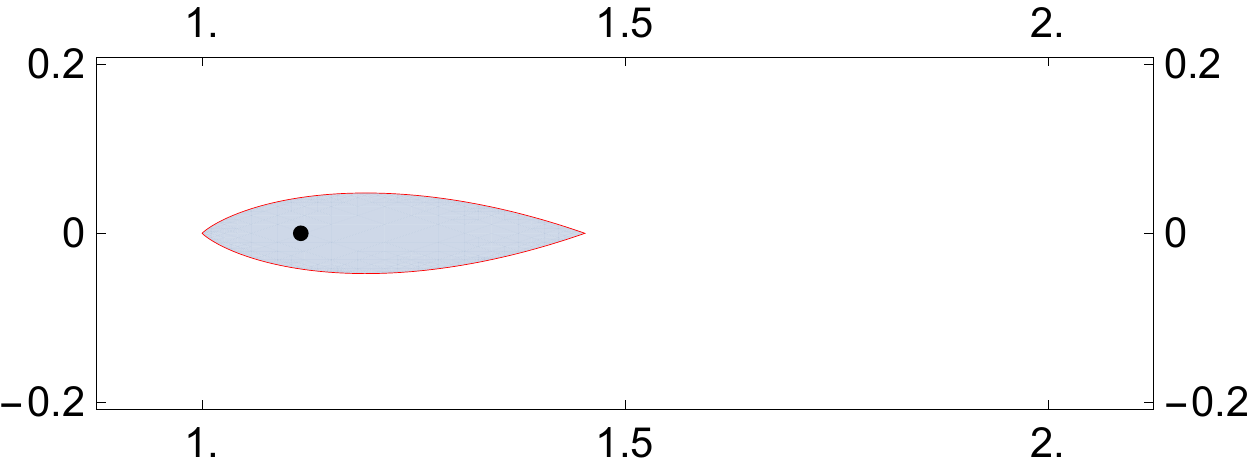} \quad 
	\includegraphics[width= 0.31 \textwidth]{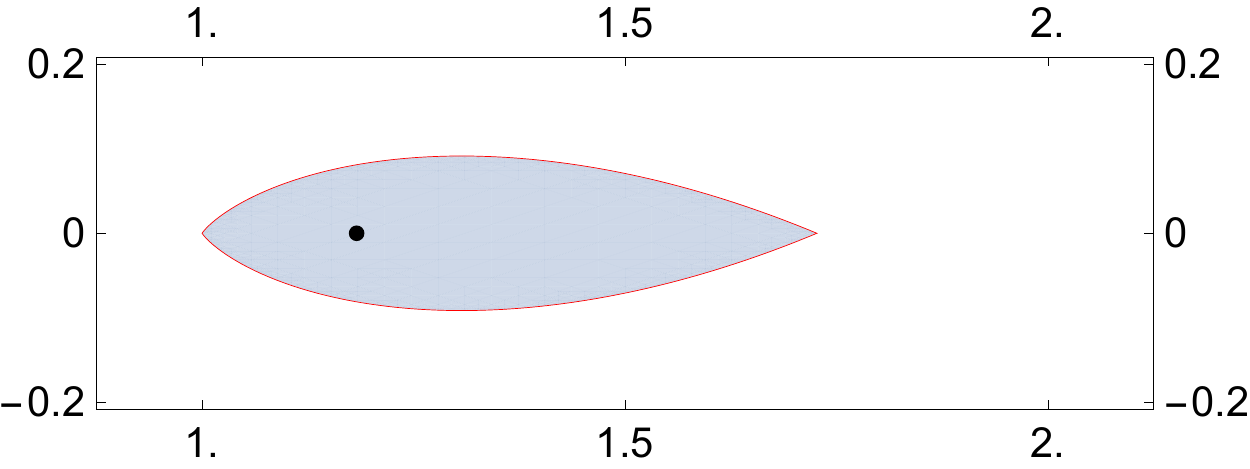} \quad
	\includegraphics[width= 0.31 \textwidth]{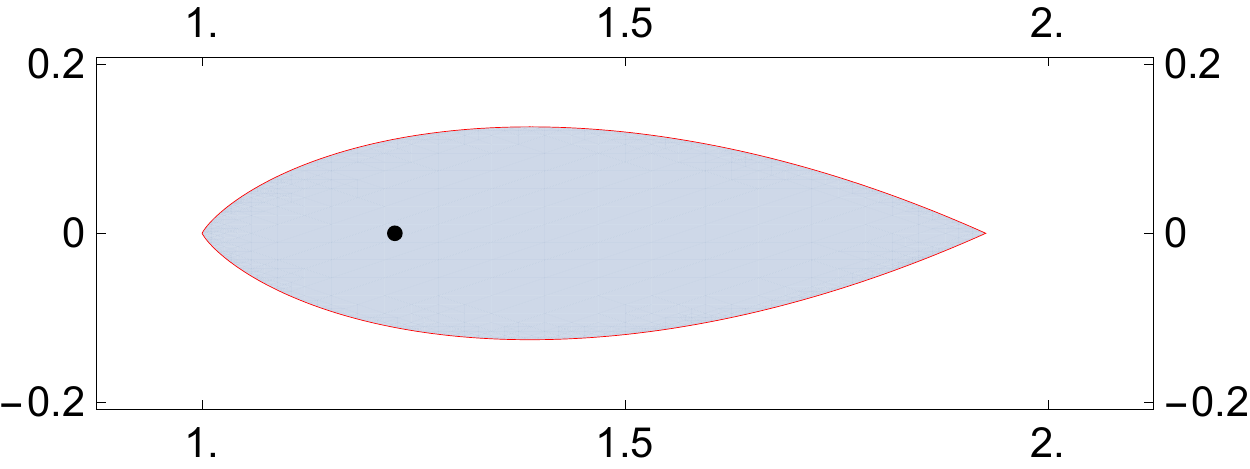}
	\caption{The region in $\C$ where $\sigma(u,\cdot) \geq 1$ for $u=0.7$, $u=0.8$ and $u=0.9$. Black ball shows the position $m^*$ of the maximum of $\sigma(u,\cdot)$.}
	\label{fig:regions}
\end{figure}
Moreover, the position $m^*(u)$ of the maximum and the maximal value of $\sigma(u,\cdot)$ are plotted in Figure \ref{fig:max}. The maximal values exceed $1$ only very little; the value of the global maximum of $\sigma(u,m)$ 
on $[0,1] \times \C$ reads approx.~$1.01038$ and it is located at approx.~$(u,m)=(0.69098,1.11015)$.
\begin{figure}[htb!]
	\includegraphics[width= 0.4\textwidth]{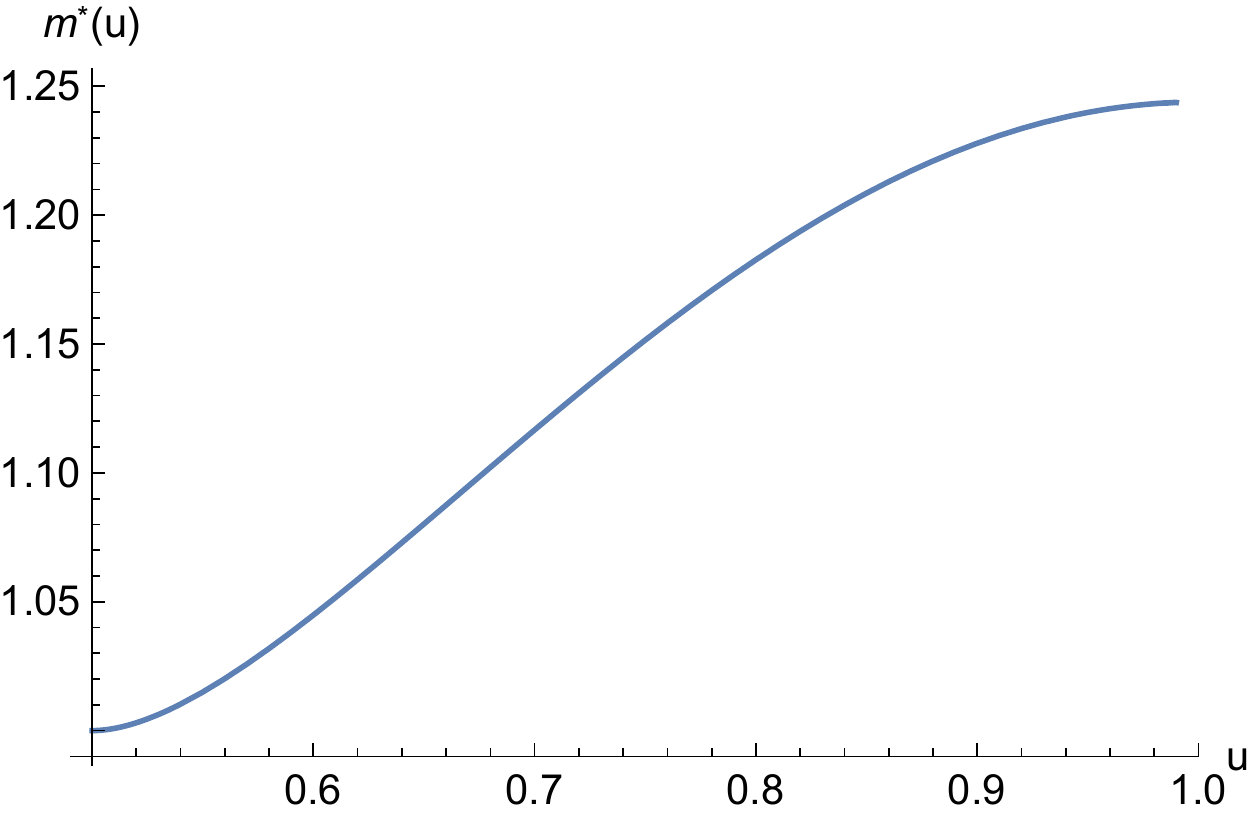} \qquad
		\includegraphics[width= 0.4\textwidth]{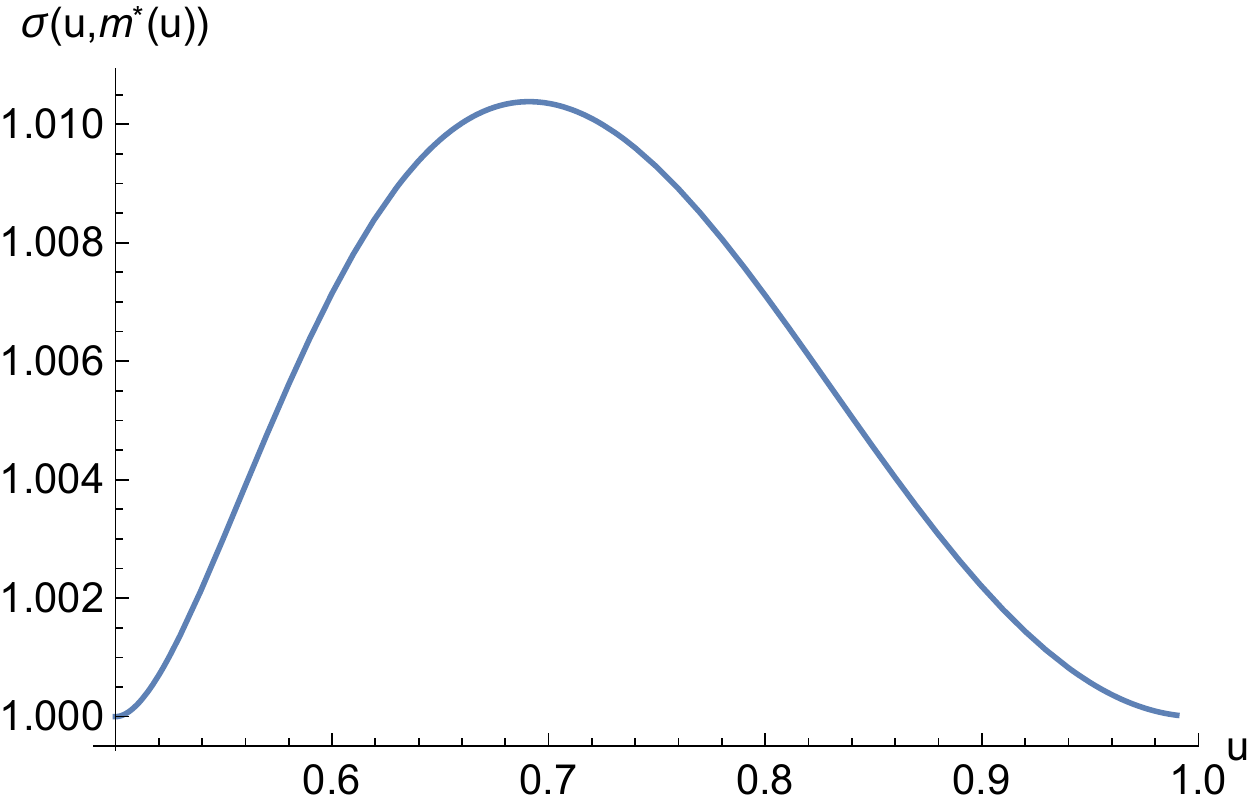}
	\caption{The position $m^*(u)$ of the maximum of $\sigma(u,m)$ (left) and the maximal value $\sigma(u,m^*(u))$ (right).}
	\label{fig:max}
\end{figure}
Finally, it even seems that $\sigma(u,\cdot)$ is convex in $(1,2)$ for $u\in(0,1/2]$ and concave for $u\in(1/2,1)$, see Figure \ref{fig:conv}.
\begin{figure}[htb!]
	\includegraphics[width= 0.6\textwidth]{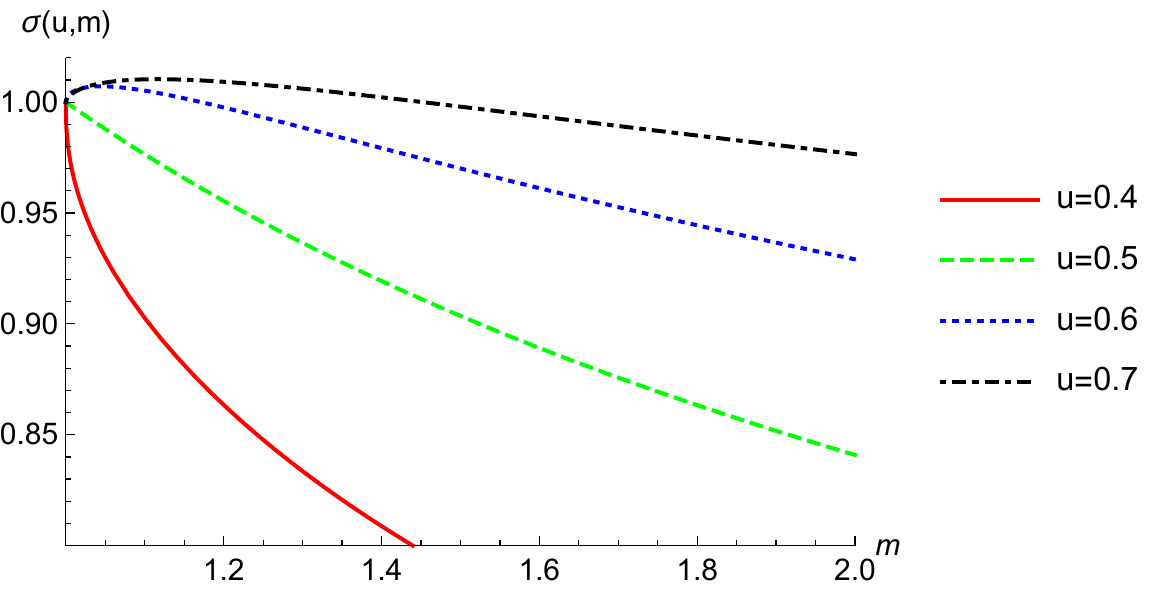}
	\caption{The plots of $\sigma(u,\cdot)$ for $u=0.4$ (red, solid), $u=0.5$ (green, dashed), $u=0.6$ (blue, dotted) and $u=0.7$ (black, dot-dashed).}
	\label{fig:conv}
\end{figure}

\section{Applications in Laplace-type integral and spectra of Jacobi matrices}
\label{sec:app}

Although the properties of the function $m\mapsto\sn(K(m)u\mid m)$ are of independent interest, 
our study was motivated by the spectral analysis of a one-parameter family of operators in $\ell^{2}(\N)$ associated with semi-infinite, in general non-symmetric, Jacobi matrices:
\begin{equation}
	J(k)=\begin{pmatrix}
		0 & 1\\
		1 & 0 & 2k\\
		& 2k & 0 & 3\\
		& & 3 & 0 & 4k\\
		& & & \ddots & \ddots & \ddots
	\end{pmatrix}, \quad k\in \overline \D,
	\label{eq:def_jacobi_mat}
\end{equation}
see~\cite{sieglstampach_inprep}.
The case $k \notin \overline \D$ can be converted to a similar problem by considering $k^{-1} J(k)$. The operator $J(k)$ turns out to be strongly linked with Jacobian elliptic functions, in particular, 
the parameter $k$ indeed coincides with the modulus, i.e. $m = k^2$. Notice that with our restriction on values of $k$, we have $m\in \overline \D$, thus we advantageously avoid the region where it is 
possible that $|\sn(K(m) u\mid m)|>1$.

While, for $k \in \D$, the spectrum of $J(k)$ is the discrete set $(2\mathbb{Z}+1)\pi/(2K(m))$, a sudden change of spectral character, sometimes called spectral phase transition, is observed when $k$ reaches 
$\partial \D$. Namely, for $k$ in the latter set, the spectrum of $J(k)$ is purely essential and, if in addition $k\neq\pm 1$, the spectrum coincides with the whole $\C$. 
The special self-adjoint and explicitly diagonalizable case $J(\pm1)$ is omitted here, see \cite{sieglstampach_inprep} for details.

The essential ingredient for the spectral results above is the analysis of functions
\begin{equation}\label{eq:CD_k_def} 
\begin{aligned}
C_{l}(z;m)&:=\int_{0}^{2K(m)}e^{-zt}\cn(t\mid m)\sn^{l}(t\mid m)\dd t,
\\
D_{l}(z;m)&:=\int_{0}^{2K(m)}e^{-zt}\dn(t\mid m)\sn^{l}(t\mid m)\dd t,
\end{aligned}
\end{equation}
where $l\in\mathbb{N}_{0}$, $z\in\C$, $m\in \overline{\D}\setminus\{1\}$, and the integration is carried out through the line segment in $\C$ connecting $0$ and $2K(m)$. The reason is that the sequence $v=v(z;m)$ 
with entries 
\begin{equation}\label{eq:def_u_k}
v_{2l+1}:=\ii(-1)^{l}m^{l/2}e^{\ii K(m)z}C_{2l}\left(\ii z;m\right),
\quad
v_{2l+2}:=(-1)^{l+1}m^{l/2}e^{\ii K(m)z}D_{2l+1}\left(\ii z;m\right),
\end{equation}
satisfies an infinite system of difference equations, compactly written as
\begin{equation}
(J(k)-z)v=-2\cos(K(m) z)e_1,
\label{eq:Ju}
\end{equation}
where $e_{1}$ denotes the first vector of the standard basis of $\ell^{2}(\N)$. 

It follows from \eqref{eq:Ju} that $z\in(2\mathbb{Z}+1)\pi/(2K(m))$ are eigenvalues of $J(k)$ if the corresponding $v(z;m)$ belongs to $\ell^{2}(\N)$. However, the latter condition is satisfied if $k \in \D$ due to the factor $m^{l/2}$ in \eqref{eq:def_u_k} and the fact that the integrals in \eqref{eq:CD_k_def} are majorized by a constant independent on $l$. Indeed, by Theorem \ref{thm:summary}, $|\sn(u\mid m))|\leq 1$ for all $m = k^{2}\in \D$ and $u\in(0,2K(m))$.

The case $k \in \partial \D \setminus \{\pm 1\}$ is more delicate. One works with singular sequences instead of eigenvectors and it is crucial to analyze the asymptotic behavior of $v_{n}$, for $n\to\infty$, leading to asymptotic expansion of integrals
\begin{equation}
	I_{l}(f)=\int_{0}^{2K(m)}f(t)\sn^{l}(t\mid m)\dd t, \quad l \to \infty,
	\label{eq:def_I_k}
\end{equation}
where $f$ is a function analytic on a neighborhood of $[0,2K(m)]$. Note that, for special choices of $f$, \eqref{eq:def_I_k} becomes \eqref{eq:CD_k_def}. The integral \eqref{eq:def_I_k} is in the form
suitable for application of the saddle point method, see, e.g., \cite[p.~417,\ Thm.~1.1]{fedoryuk87} and
\cite[Sec.~II.4]{wong01} for details. Nevertheless, we need to know maximal values of $|\sn(\cdot\mid m)|$ on $(0,2K(m))$. Thus, Theorem \ref{thm:summary} enters again, now for $m\in\partial \D\setminus\{1\}$, and we deduce that the function 
$|\sn(\cdot\mid m)|$ restricted to $(0,2K(m))$ has a unique global maximum located at $u=K(m)$ and $\sn(K(m)\mid m)=1$. 
Then the straightforward application of the saddle point method yields that, for all $m\in \overline{\D}\setminus\{1\}$, 
\[
I_{l}(f)=\frac{\sqrt{2\pi}}{2^{r}r!}\frac{f^{(2r)}(K(m))}{\left(1-m\right)^{r+1/2}}l^{-r-1/2}+O(l^{-r-1}), \quad  l\to\infty,
\]
where $r\in\N_{0}$ is the lowest integer such that $f^{(2r)}(K(m))\neq0$. In particular, one gets
\begin{align*}
C_{l}(z;m)=\frac{\sqrt{2\pi}\, z}{(1-m)e^{K(m)z}} \frac{1}{l^{3/2}}+O(l^{-2}), \quad
\ \
D_{l}(z,m)=\frac{\sqrt{2\pi}}{e^{K(m) z}}\,\frac{1}{l^{1/2}}+O(l^{-1}), \quad l\to\infty.
\end{align*}
For more details and connected spectral results on $J(k)$, see \cite{sieglstampach_inprep}.

\section*{Acknowledgements}
The research of P.~S. is supported by the \emph{Swiss National Foundation}, SNF Ambizione grant No.\ PZ00P2\_154786.
F.~{\v S}. gratefully acknowledges the kind hospitality of the Mathematisches Institut, 
Universit{\"a}t Bern; his research was also supported by grant No. GA13-11058S of the Czech Science Foundation.

\bibliographystyle{acm}
\bibliography{references}

\end{document}